\documentclass[a4paper,11pt,reqno]{amsart}
\usepackage{amssymb}
\usepackage{epsf,epsfig}
\usepackage{graphics,color}
\usepackage{amsmath}
\usepackage{amssymb}
\usepackage{cite}
\usepackage{verbatim}
\usepackage{float}
\usepackage{graphicx}
\usepackage{amsthm}
\usepackage{subfig}

\usepackage[colorlinks=true, hyperindex, pageanchor, raiselinks, bookmarks,
	 menucolor=orange!40!black,
        linkcolor=blue!40!black,
        citecolor=green!40!black,
        filecolor=magenta!40!black, 
        urlcolor=blue!40!black,
        ]{hyperref}
\hypersetup{colorlinks=true}

\usepackage{graphicx}
\usepackage{amssymb,amsmath}
\usepackage{latexsym}
\usepackage{epsfig}
\newtheorem{example}{Example}[section]

    \newcommand{\T}{\textstyle}
    \newcommand{\Z}{{\mathbb{Z}}}
\newcommand{\Y}{{\mathbb{Y}}}
\newcommand{\subdiff}{{\partial^{\rm sub}}}

\newcommand{\wto}{{\rightharpoonup}}
    \newcommand{\G}{\mathcal{I}}
    \newcommand{\R}{\mathbb{R}}
    \newcommand{\be}{\begin{eqnarray}}
    \newcommand{\ee}{\end{eqnarray}}
\newcommand{\N}{{\mathbb N}}
\newcommand{\defeq}{:=}

\newcommand{\cof}{{\rm cof}}
\renewcommand{\det}{{\rm det}}
\newcommand{\D}{\mathcal{D}}
\renewcommand{\(}{\left(}
\renewcommand{\)}{\right)}

\renewcommand{\O}{\Omega}

\newcommand{\p}{{\rm p}}
\newcommand{\e}{{\rm e}}
\renewcommand{\d}{\, \mathrm{d}}
\newcommand{\md}{\mathrm{d}}

\newcommand{\norm}[1]{\left\lVert #1\right\rVert}

\newcommand{\I}{{\mathcal{I}}}

\newcommand\gto{\mathrel{\overset{\makebox[0pt]{\mbox{\normalfont\tiny\sffamily gpc}}}{\to}}}

\newcommand{\NN}{N}

\renewcommand{\e}{{\rm e}}
\usepackage[pdftex,dvipsnames]{xcolor}
\usepackage{xargs}
\usepackage[colorinlistoftodos,prependcaption]{todonotes}
\newcommandx{\unsure}[2][1=]{\todo[linecolor=red,backgroundcolor=red!25,bordercolor=red,#1]{#2}}
\newcommandx{\change}[2][1=]{\todo[linecolor=blue,backgroundcolor=blue!25,bordercolor=blue,#1]{#2}}
\newcommandx{\improvement}[2][1=]{\todo[linecolor=Plum,backgroundcolor=Plum!25,bordercolor=Plum,#1]{#2}}
\newcommandx{\thiswillnotshow}[2][1=]{\todo[disable,#1]{#2}}

\newcommand\EEE{\color{black}}

\newcommand{\mk}{\color{black}}
\newcommand{\jz}{\color{black}}

\numberwithin{equation}{section}
\mathchardef\emptyset="001F

\newtheorem{theorem}{Theorem}[section]
\newtheorem{proposition}[theorem]{Proposition}
\newtheorem{lemma}[theorem]{Lemma}

\theoremstyle{definition}
\newtheorem{remark}[theorem]{Remark}
\newtheorem{definition}[theorem]{Definition}

\begin{document}

\keywords{elastoplasticity, energetic solution, gradient polyconvexity, rate-independent model}



\title[Elastoplasticity of gradient-polyconvex materials]{Elastoplasticity of gradient-polyconvex materials}

\author{Martin Kru\v z\'ik} 
\address[Martin Kru\v z\'ik]{Czech Academy of Sciences, Institute of Information Theory and Automation,  Pod vod\'arenskou ve\v z\'i 4, 182 08 Prague, Czechia and Faculty of Civil Engineering, Czech Technical University, Th\'{a}kurova 7, 166 29 Prague, Czechia}
\email{kruzik@utia.cas.cz}
\urladdr{http://staff.utia.cas.cz/kruzik}

\author{Ji\v{r}\'{\i} Zeman}
\address[Ji\v{r}\'{\i} Zeman]{Institut f\"{u}r Mathematik, Universit\"{a}t Augsburg, D-86135 Augsburg, Germany }
\email{jiri.zeman@math.uni-augsburg.de}
\urladdr{https://www.uni-augsburg.de/de/fakultaet/mntf/math/prof/ana/arbeitsgruppe/jiri-zeman}

\begin{abstract}
We propose a model for rate-independent evolution in elastoplastic  materials under external loading, which allows large strains. In the setting of strain-gradient plasticity with multiplicative decomposition of the deformation gradient, we prove the existence of
the so-called energetic solution. The stored  energy  density function is assumed to depend on gradients of minors of the deformation gradient  which makes our results applicable to shape-memory materials, for instance.  
\end{abstract}
\maketitle                   






\section{Introduction and notation}
Elastoplasticity at large strains is an area of ongoing research, bringing together contributions from modelling, analysis and numerical simulations. For the mathematical analysis of elastoplastic models, it is often convenient to use powerful tools from the calculus of variations, which are now able to treat quasistatic evolutionary problems as well (see e.g. \cite{ortiz-repetto}, \cite{mielke1} as pioneering works). In principle, the existence of solutions could be ensured by assuming \jz some \mk generalized convexity \EEE of the strain energy, but general material behavior may contradict this assumption. For \jz example\EEE, this is manifested in shape-memory alloys (SMA) \cite{wlsc}, magnetostrictive \cite{desimone} and ferroelectric materials \cite{shu}, even if convexity is understood in a generalized sense, such as \textit{polyconvexity} (first defined in \cite{ball77}). 

As a remedy, one can then recourse to higher-gradient regularizations, where the stored energy density $W$ also depends e.g. on the second gradient of the deformation. From a mathematical point of view, this adds compactness to the model, which is instrumental in proving the existence of solutions by the direct method. \cite{ball81} Materials with such constitutive equations are referred to as non-simple and were introduced by Toupin \cite{toupin, toupin2}. Since then, the concept has been elaborated by many authors so that its thermodynamical side is also better understood \cite{greenRivlin, capriz, podio, hooke2, forest, mierouNum, forest2}.

\textit{Gradient-polyconvex} (GPC) materials form a special class of non-simple solids and appeared in \cite{bbmkas} for the first time. Their stored energy density is not a function of the full second gradient $\nabla^2 y$ of the deformation $y\colon\Omega\to\R^3$, \mk which maps the reference configuration $\Omega$ to $\R^3$ \EEE, but it only depends, in a convex way, on the weak gradients of $\cof\nabla y$ (and that of $\det\nabla y$, if desirable). (See Section \ref{sec:main} for the definition of a cofactor matrix.) To interpret the condition physically, note that since $\det\nabla y$ measures the local change in volume between the reference and current
configuration of the material and $\cof\nabla y$ describes the transformation of area \cite[p.~78]{gurtinBook}, the stored energy $W$ depending on their gradients offers a control of how abruptly these
changes vary in space. 
Gradient polyconvexity has since been applied to the \mk evolution \EEE of SMA \cite{mkppas} and a numerical implementation of GPC \mk material models  \EEE is also available \cite{mhmk}.  \mk GPC allows us to consider stored energy densities without assuming any notion of convexity in the deformation gradient variable. This  makes it suitable for \jz modelling \mk of \jz shape-memory \mk alloys, for instance, because the resulting functional is lower semicontinuous in the underlying weak topology. An alternative approach to energy functionals that are not lower semicontinuous  is relaxation \cite{dacorogna}, however, available results prevent us from considering energies tending to infinity for extreme compression. \EEE In this article, we study an elastoplastic model using gradient polyconvexity.

The idea of strain-gradient plasticity is similar to that of nonsimple materials, as it also uses higher-order terms to prevent physical quantities from unrealistic fine-scale oscillations. Incorporating gradients of plastic variables in the constitutive equations is common in the engineering literature \cite[p.~250]{RIS} and we refer an interested reader to \cite{dk, gurtin, Gurtin-Anand, mktr-book, zpbmj, muaif, tsa-aif} or \cite{ebobisse} and the references therein. Gradient terms account for non-local interactions of dislocations and we include them as well, as they offer a suitable regularization to our model.

For our problem, we formulate the so-called \textit{energetic solution} (let us name \cite{AMFTVL} as an early reference; other related sources are cited in \cite{RIS}). One advantage of the energetic formulation is that it avoids derivatives of constitutive equations and time derivatives of the solution itself. \cite{mielke1} The variational nature of this solution concept also combines well with homogenization and relaxation. Two conditions lie at the core of the energetic formulation: a \textit{stability inequality}, which couples minimization of the elastic energy with a principle of maximum dissipation, and an \textit{energy balance}. The two requirements together imply that a usual \textit{plastic flow rule} is satisfied for sufficiently smooth solutions. 

\mk The plan of the paper is as follows: \jz in \EEE Section 2 we review some basic facts from the modelling of \textit{materials with internal variables}, loosely following \cite{fm} and \cite{han}. For the sake of completeness, we also motivate the definition of an energetic solution, although this has been done more thoroughly in previous works of Mielke {\it et al.} The main part of this paper is Section 3, where we study the rate-independent behavior of elastoplastic GPC materials under external loads and prove the existence result. The section is concluded with an example from crystal plasticity, which illustrates the usability of our findings.

Our approach draws inspiration from \cite{mami2}, where the existence of energetic solutions in large-strain elastoplasticity is proved in the presence of plastic strain gradients and a polyconvex $W$. However, as our stored energy is gradient polyconvex, our findings apply to the setting of multiwell \mk energies\EEE, as encountered e.g. in shape-memory alloys, cf. \cite{bhatta}, \cite{mierou} or \cite{krzi} for instance. We remark that energetic formulations in elastoplasticity have also been propounded: in the linear framework \cite{gl}, without strain gradients \cite{cckham}, using a finite quasiconvex energy \cite{jkmk}, involving a plastic Cauchy-Green tensor \cite{grandi1, grandi2}, for numerical computations \cite{mierouNum}, and elsewhere. The paper \cite{critplast} discusses different assumptions in quasistatic large-strain elastoplastic evolutions. Lastly, in \cite{mkdmus}, rate-independent dislocation-free plasticity is treated. 
 
 In what follows, $L^\beta(\O;\R^n)$, $1\le\beta<+\infty$ denotes the usual Lebesgue space of mappings $\O\to\R^n$ whose modulus is  integrable with the power $\beta$ and $L^\infty(\O;\R^n)$ is the space of measurable and essentially bounded mappings $\O\to\R^n$.
 Further, $W^{1,\beta}(\O;\R^n)$ standardly represents the space of mappings which live in  $L^\beta(\O;\R^n)$ and their gradients belong to $L^\beta(\O;\R^{n\times n})$. Finally, $W^{1,\beta}_0(\O;\R^n)$ is a subspace of $W^{1,\beta}(\O;\R^n)$  of maps with a zero trace on $\partial\O$.
 The weak convergence in $L^\beta(\O;\R^n)$ is defined as follows: $y_k\to y$ weakly in $L^\beta(\O;\R^n)$ (weakly star for $\beta=+\infty$) if $\int_\O y_k(x)\cdot\varphi(x)\,\md x\to\int_\O y(x)\cdot \varphi(x)\,\md x$ for all $\varphi\in L^{\beta'}(\O;\R^n)$ where $\beta'=\beta/(\beta-1)$ if $1<\beta<+\infty$, $\beta'=1$
 if $\beta=+\infty$ and $\beta'=+\infty$ for $\beta=1$.  Weak convergence of mappings and their gradients in $L^\beta$ then defines the weak convergence in $W^{1,\beta}(\O;\R^n)$. We also write \textit{w}-$\lim_{k\to+\infty}y_k=y$ or $y_k\wto y$ to denote weak convergence. Finally, $C(\O)$  or $C(\R^{n\times n})$ stand for function spaces of functions continuous on $\O$ or $\R^{n\times n}$, respectively.

 If $f\colon\R^n\to\R$ is convex but possibly nonsmooth we define its subdifferential at a point $x_0\in\R^n$ as the set of all $v\in\R^n$ such that $f(x)\ge f(x_0)+v\cdot(x-x_0)$ for all $x\in\R^n$. The subdifferential of $f$ will be denoted $\partial^{\rm sub}f$ and its elements will be called subgradients of $f$ at $x_0$.

 \bigskip
 
 \section{Motivation: modelling inelastic processes with internal variables} 

Consider $\O\subset\R^n$, a bounded Lipschitz domain representing the so-called reference configuration of a solid body, and a mapping $y\colon\O\to\R^n$, the deformation which the body is subjected to. \mk We explain the idea of elastoplasticity and the concept of energetic solutions following freely the exposition in \cite{fm} \jz and restricting ourselves to simple materials.\EEE

According to Han and Reddy \cite[p.~34]{han}, plastic deformation `is most conveniently described in the framework of materials with internal variables'. Those material models are not only governed by external (controllable) variables, such as temperature or strain, but also incorporate a vector $z\in Z\subset\R^m$ of internal variables, which describe e.g. an ongoing chemical reaction, elastoplastic
behavior or material damage. (Details can be found in \cite{HNguyen} or in more recent works on the subject listed in \cite[p.~39]{han}.) The \textit{hyperelastic} stored energy density $\mathcal{W}$ then has the form $\mathcal{W}=\mathcal{W}(F,z)$, if we consider a {\it simple material}, i.e. a material with constitutive equations involving only the first gradient $F=\nabla y$ of the deformation $y$. 

The thermodynamic conception is usually that differentiating $\mathcal{W}$ with respect to $F$, we get the mechanical stress, whereas the derivative $-\partial_z \mathcal{W}$ gives another stress-like variable – the so-called thermodynamic force 
\be
Q:=-\frac{\partial}{\partial z}\mathcal{W}(F,z)\ee
associated with the internal variable $z$. We can imagine that $Q$ tries to restructure the material irreversibly, which would lead to changing the value of $z$. The development of convex analysis \cite{moreau}  allowed quite a general formulation of an evolution rule for the internal variable $z$. Assuming the existence of a nonnegative convex potential of dissipative forces $\delta=\delta(\dot z)$, where $\dot z$ denotes the time derivative of $z$,  we write the \textit{flow law} as 
\be\label{fr} Q(t)\in\subdiff\delta(\dot z(t))\ee
everywhere in $\Omega$. A common simplifying assumption is \textit{rate-independent} behavior. It is suitable for some particular materials or for the modelling of processes with low rates of external loading. In simple terms, rate-independence means that rescaling the loading in time only results in a corresponding time-rescaling of the deformation and no additional viscous, inertial or thermal effects arise. Rate-independence translates into positive one-homogeneity of $\delta$, i.e. $\delta(\alpha\dot z)=\alpha\delta(\dot z)$ for all $\alpha>0$.

\bigskip

\begin{remark}
Since the subdifferential is \textit{monotone} (see \cite{rockafellar}), by (\ref{fr}) we have
\be
(Q(t)-\theta)\cdot(\dot z(t)-\xi)\ge 0.
\ee
for all $\theta\in\subdiff\delta(\xi)$. Choosing $\xi=0$ and observing that the one-homogeneity of $\delta$ implies $\delta(\dot z)=\omega\cdot\dot z$ for all $\omega\in\subdiff\delta(\dot z)$  we get 
\be\label{mdp}
\delta(\dot z(t))=Q(t)\cdot \dot z(t)\ge \theta\cdot\dot z(t)
\ee
for all $\theta\in\subdiff\delta(0)$.
Hence we derived the so-called {\it maximum dissipation principle} (see e.g.~ \cite{hill} or \cite{simo, simo2}) which states that the plastic dissipation that takes place in reality is not less than any possible dissipation due to thermodynamic forces ``available'' in the so-called {\it elastic domain} $\subdiff\delta(0)$.
\end{remark}

\bigskip

Hereafter, $\nu$ is the outer unit normal to $\partial\O$, and $\partial\O\supset\Gamma_0,\Gamma_1$ which are disjoint. Let $f(t)\colon\O\to\R^n$ be  the (volume) density of external \textit{body} forces
and $g(t)\colon\Gamma_1\subset\partial\O\to\R^n$ be the (surface) density of surface forces. The conservation of momentum yields the \textit{equilibrium equations}

\be\label{rovnovaha}
-{\rm div}\left(\frac{\partial}{\partial F}\mathcal{W}(\nabla y(t),z(t))\right) = f(t) \mbox{ in $\O$},
\ee
 \be\label{Bcond} y(t,x)=y_0(x) \mbox{ on $\Gamma_0$},\ee
\vspace{-2mm}
 \be\label{neumann}\frac{\partial}{\partial F}\mathcal{W}(\nabla y(t),z(t)) \nu(x) =g(t,x) \mbox{ on $\Gamma_1$}.\ee

If $z_0\in Z$ is an initial condition for the internal variable, the system of equations (\ref{rovnovaha})-(\ref{neumann}) together with (\ref{fr1}): 

 \be\label{fr1}-\frac{\partial}{\partial z }\mathcal{W}(\nabla y(t),z(t))\in\partial\delta(\dot z(t)),\ z(0)=z_0,\ee
governs the mechanical behavior characterized by the unknowns $y(t)$, $z(t)$.

Unfortunately, the system (\ref{rovnovaha})-(\ref{neumann}) is ill-posed in many situations. \mk See  the works of Suquet \cite{suquet} and Temam \cite{temam} for early analysis of this problem.\EEE 

To get existence results, it seems \mk  necessary \EEE to include the gradient of the internal variable, i.e. use the strain energy density
$$\tilde{\mathcal{W}}(\nabla y, z,\nabla z):=\mathcal{W}(\nabla y, z)+\epsilon|\nabla z|^\alpha$$
for $\alpha\ge 1$ and $\varepsilon>0$. Let us briefly recall how an energy balance, which appears in the \textit{energetic} formulation of this regularized problem, is obtained, under sufficient smoothness and integrability assumptions on all the present mappings. For details, see \cite{fm,mielke1}.

The functional
\be \mathcal{I}(t,y(t),z(t)):=\int_\O \mathcal{W}(\nabla y(t),z(t))\,\md x+\epsilon\int_\O |\nabla z(t)|^\alpha\,\md x-L(t,y(t)),\ee
expresses the potential energy in our system, where the work done by external forces is
\be\label{loading1}
L(t,y(t)):=\int_\O f(t)\cdot y(t)\,\md x +\int_{\Gamma_1} g(t)\cdot y(t)\,\md S.\
\ee

We also introduce the total dissipation along $z$,
$$
{\rm Diss}(z;[0,t]):=\int_0^t\int_\O\delta(\dot z(s))\,\md x\md s.$$

Calculating the thermodynamic force $\partial_z \tilde{\mathcal{W}}$, using the relation $\delta(\dot z)=\omega\cdot\dot z$, $\omega\in\subdiff\delta(\dot z)$, mentioned above (\ref{mdp}), we can deduce from (\ref{rovnovaha})-(\ref{neumann}) the following energy balance:
$$
\mathcal{I}(t,y(t),z(t))+{\rm Diss}(z;[0,t])=\mathcal{I}(0,y(0),z(0))+\int_0^t\dot L(s,y(s))\,\md s.$$

Due to low regularity in time, a more general expression for ${\rm Diss}(z;[0,t])$ must be used in practice, though – see (\ref{e-bal}) below.

To this end, we define a dissipation distance between two values $z_0,z_1\in \Z$ of the internal variable: 
\be\label{nodist}
D(x,z_0,z_1):=\inf_z\left\{\int_0^1 \delta(x,z(s),\dot z(s))\,\md s;\ z(0)=z_0,\ z(1)=z_1\right\},\ee
where $z\in C^1([0,1];\Z)$, and set 
\be\label{dist}
\D(z_1,z_2)=\int_\O D(x,z_1(x),z_2(x))\,\md x
\ee
for $z_1,z_2\in \mathbb{Z}:=\{z\colon\O\to\R^m;\ z(x)\in Z\mbox{ a.e. in $\O$}\}$, as in \cite{mielke1}.

In order to find a quasistatic evolution of the system Mielke, Theil, and Levitas \cite{AMFTVL} came up with the following definition of the energetic solution which conveniently overcomes nonsmoothness mentioned above.  Moreover, it fully exploits the possible variational structure of the problem and allows for very general 
energy and dissipation functionals. This concept has versatile applications to many problems in continuum mechanics of solids. Additionally,  working with $\I$ and $\D$ directly allows us to include also higher derivatives of $y$ into the model  or to require integrability of some functions of $\nabla y$ as also done in this contribution.

\subsection{Energetic solution}
Suppose that the evolution of $y(t)\in \Y$ and $z(t)\in \Z$ is studied during a time interval $[0,T]$.
The following two properties characterize the energetic solution due to Mielke {\it et al.} \cite{AMFTVL}.

\noindent
(i) Stability inequality:\\
$\forall t\in[0,T],\, \tilde z\in \Z,\, \tilde{y}\in \Y$:
\be\label{stblt}\mathcal{I}(t,y(t),z(t))\le \mathcal{I}(t,\tilde y,\tilde z)+\mathcal{D}(z(t),\tilde z)\ee

\noindent
(ii) Energy balance:
$\forall\ 0\le t\le T$

\be\label{e-bal}\mathcal{I}(t,y(t),z(t))+{\rm Var}(\mathcal{D},z;[0,t]) =\mathcal{I}(0,y(0),z(0)) +\int_0^t \dot L(\xi,y(\xi))\,\md \xi, \ee

$$\text{where }
{\rm Var}(\mathcal{D},z;[s,t]):=\sup\left\{\sum_{i=1}^N \mathcal{D}(z(t_i),z(t_{i-1}));\ \{t_i\} \mbox{ partition of } [s,t]\right\}.$$  
  
\begin{definition}\label{en-so}   
The mapping $t\mapsto(y(t),z(t))\in \Y\times \Z $ is an energetic solution to the problem
$(\mathcal{I},\delta, L)$ if the stability inequality and energy balance 
are satisfied for all $t\in [0,T]$.   
\end{definition} 

\bigskip

\begin{remark}
The mechanical idea behind stability inequality (i) is the following: imagine first that $\tilde{z} := z(t)$,
then $\D(z(t),\tilde{z})$ vanishes, since no change in the internal variables implies no dissipation.
Consequently, (i) simplifies to $\mathcal{I}(t,y(t),z(t))\le \mathcal{I}(t,\tilde y,\tilde z)$ for all $\tilde{y}\in \Y$ and $y(t)$ is a
global minimizer of $\mathcal{I}(t,\cdot, z(t))$ over $\Y$. So we see that in this case, (i) has the meaning of an elastic equilibrium. If $\tilde{z}\neq z(t)$, then the amount of dissipated
energy between the states $\tilde{z}$ and $z(t)$ must, by (i), at least compensate for, if not outweigh the associated loss in the total energy, which is a version of the principle of maximum dissipation. \cite{mielke1} 
\end{remark}

We will write $\mathbb{Q}:=\mathbb{Y}\times \mathbb{Z}$  and set $q:=(y,z)$.
Next let us define the set of \textit{stable states} at time $t$ as 

\be
\mathcal{S}(t):=\{q\in\mathbb{Q};\ \forall \tilde q\in\mathbb{Q}:\  \mathcal{I}(t,q)\le \mathcal{I}(t,\tilde q)+\mathcal{D}(q,\tilde q)\}
\ee
and 
\be\mathcal{S}_{[0,T]}:=\bigcup_{t\in[0,T]}\{t\}\times\mathcal{S}(t).\ee

Moreover, a sequence $\{(t_k,q_k)\}_{k\in\N}$  is called {\it stable} if $q_k\in \mathcal{S}(t_k)$.

\bigskip

\section{Applications to elastoplasticity}\label{sec:main}
This section shows how the energetic approach can be applied to an elastoplastic problem of gradient-polyconvex materials.

\subsection{Gradient polyconvexity}
Gradient polyconvexity was first defined in \cite{bbmkas} in analogy with classical polyconvexity \cite{ball77}. The difference is that here we assume that the stored energy density is convex in gradients of minors of the deformation gradient but not in minors alone. More precisely, the following definition is taken from \cite{bbmkas}. We recall that for an invertible $F\in\R^{n\times n}$ we  define the cofactor of $F$, $\cof F=(\det F) F^{-\top}\in\R^{n\times n}$. \mk As  is shown in \cite{bbmkas} there are maps $y\in W^{1,1}(\Omega;\R^3)$ such that $\det\nabla y$ and $\cof\nabla y$ are Lipschitz continuous but $y\not\in W^{2,1}(\Omega;\R^3)$. The same conclusion can be reached for every $n\ge 3$.  On the other hand, if $n=2$ then $\nabla\cof\nabla y$ has the same entries (up to the minus sign) as $\nabla^2 y$.  \EEE

\begin{definition} \label{def-gpc}
Let $\O\subset\R^n$ be a bounded open domain. Let $\mathcal{W}_1\colon\R^{n\times n}\times\R^{n\times n\times n}\to\R\cup\{+\infty\}$ be a lower semicontinuous function. The functional
\begin{align}\label{full-I}
J(y) = \int_\Omega \mathcal{W}_1(\nabla y(x), \nabla[ \cof \nabla y(x)]) \d x,
\end{align}
defined for any measurable function \mk  $y\colon \Omega \to \R^n$ for which the weak derivatives $\nabla y$, $\nabla[ \cof \nabla y]$ exist and are integrable is called {\em gradient polyconvex} if the function $\mathcal{W}_1(F,\cdot)$ is convex for every  $F\in\R^{n\times n}$. \EEE
\end{definition}

 We assume that for some $c>0$, and  numbers \jz $\alpha> n-1$\EEE, $s>0$ it holds that  for every $F\in\R^{n\times n}$ and every $H\in\R^{n\times n\times n}$ we have that
\begin{align}\label{growth-graddet1}
\mathcal{W}_1(F,H)\ge\begin{cases}
c\big(|F|^\alpha  + (\det F)^{-s}+|H|^{\alpha/(n-1)} \big)
&\text{ if }\det F>0,\\
+\infty&\text{ otherwise,} \end{cases}
\end{align}
where $|\cdot|$ denotes the Euclidean norm. These growth assumptions can be surely weakened but we stick to them in order to simplify our presentation.

The idea behind condition \eqref{growth-graddet1} is that the energy blows up if the deformation does not preserve orientation or if the measures of strain on the right-hand side grow to extreme values.  Coercivity conditions involving $|F|$ and $\det F$ are commonly used in  nonlinear elasticity (see e.g. \cite{ciarlet}) and reflect variations in volume or changes of the distances of points caused by the deformation. The term  $H$ which is a placeholder for $\nabla\cof\nabla y$ penalizes spatial changes of $\cof\nabla y$ and, consequently, it aims at suppressing abrupt areal changes in the deformed configuration. 

An important feature of gradient polyconvexity is that {\it no} convexity assumptions on $\mathcal{W}_1$ are needed in the $F$-variable, so that very general material laws can be considered including multiwell energy functions \cite{bhatta}  or the St.~Venant–Kirchhoff energy density \cite{ciarlet}.  We call materials whose  stored $J$ energy obeys  \eqref{full-I} also gradient-polyconvex. 

\subsection{Assumptions on problem data}\label{assump}

As in \cite{gurtin}  we will consider  so-called {\it separable materials}, i.e. materials where the elastoplastic energy density 
has the form 
\be\label{separable}
\mathcal{W}(F_\e, H, F_p,\nabla F_\p, p,\nabla p):=\mathcal{W}_1(F_\e, H)+\mathcal{W}_2(F_p,\nabla F_\p, p,\nabla p).
\ee   

We will assume that $\mathcal{W}_1$ is continuous and satisfies \eqref{growth-graddet1} while for $\mathcal{W}_2$ we require:\\
\noindent
(i) The plastic part $\mathcal{W}_2$ is continuous in its  all  arguments. 

\noindent
(ii) Suppose that  there are two constants $C,c>0$ so that the  following assumption holds for constants $c_1>0$,
$\beta>n$, and $\omega>n$:

\begin{align}\label{growth}
C(1+|F_{\rm p}|^{\beta}+|G|^{\beta} +|p|^{\omega}+|\pi|^{\omega})&\ge\mathcal{W}_2(F_\p, G, p,\pi)\nonumber\\
&\ge c(|F_\p|^{\beta}+|G|^{\beta} +|p|^{\omega}+|\pi|^{\omega})-c_1.
\end{align}

\noindent
(iii) There is $c_2>0$, $v^*\in\R^m$ and a modulus of continuity \jz $\hat{\omega}$ such that for all $\hat{\alpha}>0$ \EEE, $F_{\rm p}\in\R^{n\times n}$, $G\in\R^{n\times n\times n}$, $p\in\R^m$ and $\pi\in\R^{m\times n}$:
\be\label{W2cont}
|\mathcal{W}_2(F_\p,G,p+\hat{\alpha} v^*,\pi)-\mathcal{W}_2(F_\p,G,p,\pi)|\le \hat{\omega}(\hat{\alpha})(\mathcal{W}_2(F_\p,G,p,\pi)+c_2).
\ee

Furthermore, let us suppose that for every $F_{\rm e}$, $F_{\rm p}\in\R^{n\times n}$ and $p\in \R^m$, the functions $\mathcal{W}_1(F_{\rm e},\cdot)$ and $\mathcal{W}_2(F_{\rm p},\cdot,p,\cdot)$ are convex.

 The dissipation distance $\D\colon\Z\times\Z\to[0,+\infty]$ takes the form (\ref{dist}) \jz for a function $D\colon \O\times ({\rm SL}(n)\times\R^m)^2$ \EEE and we only change the definition of $\Z$ to (\ref{Z}) below. We make the following assumptions on $\mathcal{D}$:

\noindent
(i) Lower semicontinuity:
\be\label{assonD1}
\mathcal{D}(z,\tilde z)\le\liminf_{k\to\infty}\mathcal{D}(z_k,\tilde z_k),\ee whenever $z_k\wto z$ and $\tilde z_k \wto\tilde z$. 

\noindent
(ii) Positivity:
\be\label{assonD2}
\text{If }\{z_k\}\subset Z\text{ is bounded and }\min\{\mathcal{D}(z_k,z),\mathcal{D}(z,z_k)\}\to 0\text{ then }z_k\wto z.
\ee

\noindent
(iii) For all $z_1$, $z_2\in\mathbb{Z}$: $\mathcal{D}(z_1, z_2)=0$ if and only if $z_1=z_2$.

\noindent
(iv) Triangle inequality: $\mathcal{D}(z_1, z_3)\le\mathcal{D}(z_1, z_2)+\mathcal{D}(z_2,z_3)$ for all $z_1$, $z_2$, $z_3\in\mathbb{Z}$.

We refer the reader to \cite{mami2} to see that (ii) follows from (i) and (iii). \jz After stating Proposition \ref{limstab}, we specify further assumptions on $D$ ((\ref{nodi1}) or (3.A)--(3.C)). Besides, it is naturally required that $D$ be such that (i)--(iv) holds.\EEE

In order to prove the existence of a solution to (\ref{timestep}) we must impose some data qualifications.
In what follows, we assume that 
\be\label{ass-f}f \in C^1\left([0, T]; L^{\tilde d} \left(\O; \R^n \right) \right),\ee
 \be\label{ass-g} g \in C^1\left([0, T]; L^{\hat d} \left(\Gamma_1; \R^n \right) \right),\ee
where \jz $\tilde d\ge [nd/(n-d)]'=nd/(nd-n+d)$ \EEE if $1\le d<n$ or $\tilde d> 1$ otherwise. 
Similarly, we suppose that \jz $\hat d\ge [(nd-d)/(n-d)]'=(nd-d)/(nd-n)$  \EEE if $d<n$ or $\hat d> 1$ otherwise.

\subsection{Formulation of the problem}
From now on, $y\colon\O\to\R^n$ will represent the deformation of a material body, whose reference configuration is a bounded Lipschitz domain $\O\subset\R^n$. Since $y$ models both \textit{elastic} and \textit{plastic} behaviour, we split the deformation gradient $F=\nabla y$ as $F=F_\e F_\p$, where $F_\e$ stands for an elastic part and $F_\p\in {\rm SL}(n):=\{A\in\R^{n\times n};\ {\rm det}\ A=1\}$ is a plastic part, which irreversibly transforms the material. To capture e.g. \textit{back stresses}, we use the vector $p\in\R^m$ of hardening internal variables. Written together, $z(x)=(F_\p(x),p(x))$ is a plastic variable, lying in ${\rm SL}(n)\times\R^m$ for almost all $x\in\O$.

The energy functional $\I$ is given by 
\be
\I(t,y(t),z(t)):=\int_\O \mathcal{W}(\nabla y F_\p^{-1}, \nabla[(\cof\nabla y)F_\p^\top], F_\p, \nabla F_\p, p,\nabla p)\,\md x-L(t,y(t)),
\ee
with $L$ defined in (\ref{loading1}).

Our stored energy density $\mathcal{W}$ does not explicitly depend on the spatial variable $x$, but treating the inhomogeneous case would not need many modifications.   

Let us remark that $(\cof \nabla y)F_\p^\top$ is the cofactor of the elastic part $F_\e$, since by the product rule for cofactor matrices \cite[p.~4]{ciarlet} and by $F_\p\in{\rm SL}(n)$, we have
$$\cof F_\e = \cof(FF_\p^{-1})=(\cof F)\cof(F_\p^{-1})=(\cof F)\det(F_\p^{-1})(F_\p^{-1})^{-\top}=(\cof F)F_\p^\top.$$
    
The admissible deformations $y$ lie in 
    $$\mathbb{Y}:=\{y\in W^{1,d}(\O;\R^n);\ y=y_0 \mbox{ on }\Gamma_0 \},
$$
    where $\Gamma_0\subset\partial\O$ with a positive surface measure and $y_0\in W^{1-1/d,d}(\Gamma_0;\R^n)$ is given. Assuming that $\Gamma_1\subset\partial\O$ as in Section 2, we suppose $\Gamma_0\cap\Gamma_1=\emptyset$.
For the internal states $z$ let us define the set 
\be\label{Z}
\mathbb{Z}:=\{(F_\p,p)\in W^{1,\beta}(\O;\R^{n\times n})\times W^{1,\omega}(\O;\R^m):\  F_\p(x)\in {\rm SL}(n) \mbox{ for a.e.~$x\in\O$}\}.
\ee

For ease of notation, we write $q=(y,z)\jz \in \mathbb{Q}=\mathbb{Y}\times\mathbb{Z}\EEE$ and understand \jz $\mathcal{I}(t,\cdot)$, $L(t,\cdot)$ and $\mathcal{D}$ \EEE as functions of $q$, that is:
\begin{gather*}
\I(t,q(t))=\int_\O \mathcal{W}(\nabla y F_\p^{-1}, \nabla[(\cof\nabla y)F_\p^{\top}], F_\p, \nabla F_\p, p,\nabla p)\,\md x-L(t,q(t)),\\
L(t,q(t):=L(t,y(t)),\\
\mathcal{D}(q_1,q_2):=\mathcal{D} (z_1, z_2)
\end{gather*}
if $q_1=(y_1,z_1)$ and $q_2=(y_2,z_2)$.

 In order to prove the existence of an energetic solution to our problem we will need the following results of technical nature.

\bigskip

\subsection{Auxiliary results}

We start this section by the following reverse Young inequality.

\begin{lemma}\label{Yconseq}
Suppose that $a>0$, $b>0$, $\delta>0$, $r>1$. Then 
\begin{equation*}
\frac{a}{b}\geq r\delta^\frac{r}{r-1}a^\frac{1}{r}-(r-1)\delta^\frac{r^2}{(r-1)^2}b^\frac{1}{r-1}.
\end{equation*} 
\end{lemma}
\begin{proof}
Young's inequality  states that given a pair of positive numbers $\alpha, \beta$ and $1<p,q<+\infty$, $\frac{1}{p}+\frac{1}{q}=1$, then
\begin{equation}\label{eq:Young}
\alpha\beta\leq\frac{\alpha^p}{p}+\frac{\beta^q}{q}.
\end{equation}
Set $p=r$, $\alpha=a^\frac{1}{r}$, $\beta=\delta^\frac{r}{r-1}b$ in \eqref{eq:Young}. Then $q=\frac{r}{r-1}$ and Young's inequality yields
$$\frac{r-1}{r}\delta^\frac{r^2}{(r-1)^2}b^\frac{r}{r-1}+\frac{1}{r}a\geq \delta^\frac{r}{r-1}a^\frac{1}{r}b,$$
which after multiplying by $\frac{r}{b}$ implies the desired result.
\end{proof}

It will also be useful to give a name to a kind of convergence which makes $\I(t,\cdot)$ lower semicontinuous.

\begin{definition}
We say that the sequence \mk  $\{q_k\}_{k\in\N}\subset\mathbb{Q}$, \EEE $q_k=(y_k,z_k)=(y_k,F_{{\rm p}k},p_k)$ \emph{gpc-converges} to $q_*=(y_*,z_*)=(y_*,F_{{\rm p}*},p_*)$ if $z_k\wto z_*$ in $\mathbb{Z}$, $\nabla\cof(\nabla y_k F_{{\rm p}k}^{-1})\wto \nabla\cof(\nabla y_* F_{{\rm p}*}^{-1})$ in $L^{\alpha/(n-1)}(\O;\R^{n\times n\times n})$ and $\nabla y_k\to \nabla y_*$ in measure. We write $q_k\gto q_*$ for short.
\end{definition}

\begin{lemma}
Let $t_k\to t_*$ with $t_k$, $t_*\in [0,T]$, $k\in\N$, and $q_k\gto q_*$, $\{q_k\}_{k\in\N}\subset\mathbb{Q}$. Then $\I(t_*,q_*)\le\liminf_{k\to\infty}\I(t_k,q_k)$.
\end{lemma}
\begin{proof}
This is an immediate consequence of \cite[Corollary 7.9]{FL}. \jz Note that we can construct a subsequence such that $\nabla y_{k_j}F_{{\rm p}k_j}^{-1}\to\nabla y_* F_{{\rm p}*}^{-1}$ almost everywhere.\EEE
\end{proof}
Even though we do not have the weak lower semicontinuity of $\I$ in general, we can get it for a subsequence provided the respective values of $\I$ are bounded.
\begin{lemma}\label{gtosubseq}
\jz Provided that $\alpha^{-1}+\beta^{-1}\le d^{-1}< (n-1)^{-1}$ and $d>\frac{\beta(n-1)}{\beta-1}$\EEE, let $t_k\in [0,T]$, $k\in\N$, and $q_k\wto q_*$ in $\mathbb{Q}$. Suppose there is $C_I>0$ such that for all $k\in\N$ the bound $\I(t_k,q_k)\le C_I$ holds true. Then there exists a subsequence $\{q_{k_j}\}_{j\in\N}$ of $\{q_k\}_{k\in\N}$ that gpc-converges to the same limit $q_*$.
\end{lemma}
\begin{proof}
The proof goes the same way as in Proposition \ref{timestep-ex}. To keep the flow of ideas uninterrupted there, we postpone the presentation to that section. 
\end{proof}

\begin{proposition}\label{limstab}
Let $\I$ be lower semicontinuous with respect to gpc-convergence and let (\ref{ass-f}) and (\ref{ass-g}) hold. 
Let it for all  $(t_*,q_*)\in[0,T]\times\mathbb{Q}$ and all stable sequences  $\{(t_k,q_k)\}_{k\in\N}$  such that w-$\lim_{k\to\infty} (t_k,q_k)=(t_*,q_*)$
be true that for all $\tilde q\in\mathbb{Q}$ there is $\{\tilde q_k\}\subset \mathbb{Q}$ such that 
\begin{align}\label{limsup}
\limsup_{k\to\infty}(\I(t_k,\tilde q_k)+\mathcal{D}(q_k,\tilde q_k))
\le \I(t_*,\tilde q)+\mathcal{D}(q_*,\tilde q)).
\end{align}
Then for any stable sequence $\{(t_k,q_k)\}_{k\in\N}$ such that w-$\lim_{k\to\infty} (t_k,q_k)=(t_*,q_*)$ and $\I(t_k,q_k)\le C$ for some $C>0$, we have $\lim_{k\to\infty}\I(t_k,q_k)=\I(t_*,q_*)$ and $q_*\in\mathcal{S}(t_*)$.
\end{proposition}

\bigskip

{\it Proof.} We follow the proof of Prop.~4.3 in \cite{mami2}.
Take $\tilde q:=q_*$ in (\ref{limsup}), which yields a sequence $\{\tilde q_k\}_{k\in\N}$. Then we get, by the stability of $q_k$,
\be\label{Iusc}
\limsup_{k\to\infty}\I(t_k, q_k)\le \limsup_{k\to\infty}((\I(t_k,\tilde q_k)+\mathcal{D}(q_k,\tilde q_k))\le \I(t_*,\tilde q)+\mathcal{D}(q_*,\tilde q)= \I(t_*,q_*).
\ee

The assumptions (\ref{ass-f}) and (\ref{ass-g}) on $f$ and $g$ further give
\be\label{tStartk}
\lim_{k\to\infty}|\I(t_k,q_k)-\I(t_*,q_k)|=\lim_{k\to\infty}|L(t_k,q_k)-L(t_*,q_k)|=0.
\ee

Since $\I$ is lower semicontinuous with respect to gpc-convergence (and by Lemma \ref{gtosubseq} we can pass to a gpc-convergent subsequence, without relabeling it), we deduce by equation (\ref{tStartk}) that 
$$\liminf_{k\to\infty} \I(t_k,q_k) =\lim_{k\to\infty}(\I(t_k,q_k)-\I(t_*,q_k))+\liminf_{k\to\infty}\I(t_*,q_k)\ge\I(t_*,q_*).$$
This combined with (\ref{Iusc}) establishes the weak continuity along a stable sequence: $\I(t_k,q_k)\to\I(t_*,q_*)$.
In the end, pick a $\tilde q\in\mathbb{Q}$ and apply (\ref{limsup}) to it: 
$$
\I(t_*,q_*)=\lim_{k\to\infty}\I(t_k,q_k)\le\liminf_{k\to\infty}(\I(t_k,\tilde q_k)+\mathcal{D}(q_k,\tilde q_k))\le \I(t_*,\tilde q)+\mathcal{D}(q_*,\tilde q);$$
hence, the stability of $q_*$ is proved.
\hfill
$\Box$

A natural question is how to ensure the validity of (\ref{limsup}).
If $\mathcal{D}\colon\mathbb{Q}\times\mathbb{Q}\to [0,+\infty)$, i.e. no irreversibility constraint is imposed on plastic processes, then it is  sufficient if $D$ from (\ref{dist}) satisfies
\jz \be\label{nodi1}\begin{aligned}
&D\text{ is a Carathéodory mapping and}\\
&D(x,z_1,z_2)\le c(x)+ C(|F_{\p1}|^{\beta^*-\epsilon} + |F_{\p2}|^{\beta^*-\epsilon}+|p_1|^{\omega^*-\epsilon}+|p_2|^{\omega^*-\epsilon}),
\end{aligned}\ee\EEE
where $\epsilon>0$  is small enough and $\beta^*:=n\beta/(n-\beta)$ if $n>\beta$
and $\beta^*>1$ if $\beta\ge n$. Similarly, $\omega^*:=n\omega/(n-\omega)$ if $n>\omega$
and $\omega^*>1$ if $\omega\ge n$.  
Then the compact embedding provides the continuity of $\mathcal{D}$. This shows that (\ref{limsup}) is valid with a constant sequence $\tilde{q}_k=\tilde{q}$.

If $\mathcal{D}\colon\mathbb{Q}\times\mathbb{Q}\to [0,+\infty]$,  the assumptions are more elaborate. Following \cite{mami2} we impose the following {\it sufficient} conditions on $D$ from (\ref{nodist}):\\
\noindent
\jz (3.A) \EEE $D(x,\cdot,\cdot)\colon\mathbb{D}(x)\to[0,+\infty)$ is continuous, where
 $\mathbb{D}(x):=\{(z_1,z_2);\ D(x,z_1,z_2)<+\infty\}$,\\
\noindent
(3.B) For every $R>0$ there is $K>0$ such that for almost all $x\in\O$:
$D(x,z_1,z_2)<K$ if $z_1,z_2\in \mathbb{D}(x)$ and $|z_1|,|z_2|<R$, and \\
(3.C) The direction $v^*\in\R^m$ from (\ref{W2cont}) has the property that for all $\alpha,R>0$ there is $\rho>0$ such that  for almost every $x\in\O$ and every $z,z_0,z_1$:
$$ |z-z_0|<\rho \mbox{, }(z_0,z_1)\in \mathbb{D}(x)\mbox{ and }|z_0|,|z_1|<R\mbox{ implies } (z,z_1+(0,\alpha v^*))\in \mathbb{D}(x).$$

\bigskip

\begin{proposition}
Let $\beta,\omega>n$.
Let $D$ satisfy (3.A)--(3.C). Then (\ref{limsup}) holds.
\end{proposition}  

\bigskip

{\it Proof.} The reasoning follows the lines of \cite{mami2}.
If $\mathcal{D}(q_*,\tilde q)=+\infty$ in (\ref{limsup}), the proof is finished.
So, we assume that $$\mathcal{D}(q_*,\tilde q)\in\R. $$
If $q_k\wto q_*$, we observe that 
\be\label{compPlast}
\rho_k:=\|F_{\p k}-F_{\p*}\|_{C(\bar\O;\R^{n\times n})}+\|p_k- p_*\|_{C(\bar\O;\R^m)}\to 0.
\ee
by the compact embedding. Then $|z_k|+|z_*| +|\tilde z|<R$ for some $R>0$ if $k$ is large enough. Define $\tilde z_k:= (\tilde F_\p,\tilde p+\alpha_k v^*)$ where $\alpha_k\to 0$ and  relates to $\rho_k$ 
as in (3.C) (we may need to redefine the $\rho_k$ from (\ref{compPlast}) by passing to a subsequence, which is without loss of generality). Thus, $(z_k,\tilde z_k)\in \mathbb{D}(x)$ a.e. in $\O$ and we have $|z_k|$, $|\tilde z_k| <R$. The continuity of $D$ gives the convergence of 
$D(x,z_k,\tilde z_k)\to D(x,z_*,\tilde z)$ pointwise so that $\mathcal{D}(q_k,\tilde q_k)\to\mathcal{D}(q_*,\tilde q)$ by condition (3.B)
and the dominated convergence theorem. Furthermore, properties of $\mathcal{W}_2$ and $L$ (assumptions (\ref{ass-f}), (\ref{ass-g})) imply that $\I(t_k,\tilde{q}_k)\to \I(t_*,\tilde{q}_*)$. Summing up, we deduce that (\ref{limsup}) is satisfied with equality.
 \hfill
 $\Box$

\subsection{Incremental problems}

Next, we define the following sequence of incremental problems. We consider a \jz {\it stable} \EEE  initial condition $q^0_\tau:=q^0\in\mathbb{Q}$.

Let us take $\tau>0$, a time step, chosen in the way that $N=T/\tau\in\N$.
For $1\le k\le N$, $t_k:=k\tau$,  find $q_\tau^k\in\mathbb{Q}$ such that 
 $q^k_\tau$ solves
\be\label{timestep}\left.\begin{array}{ll}
\mbox{minimize } & \I(t_k,q)+\mathcal{D}(q_\tau^{k-1},q)\\
\mbox{subject to } & q^k_\tau\in\mathbb{Q}.
 \end{array}\right\}\ \ee

\bigskip

\begin{proposition}\label{timestep-ex}
Let $\alpha^{-1}+\beta^{-1}\le d^{-1}< \jz (n-1)^{-1}\EEE$, $\jz d>\frac{\beta(n-1)}{\beta-1}\EEE$. Let the assumptions on $\mathcal{W}$  and $\mathcal{D}$ be satisfied. 
Let further (\ref{ass-f}) and (\ref{ass-g}) be satisfied. Then the problem (\ref{timestep}) has a solution for all $k=1,\ldots, T/\tau$. In addition, for the solution $q_\tau^k=(y_\tau^k,z_\tau^k)$ we get $\det\nabla y_\tau^k>0$ a.e. in $\O$.
\end{proposition}

\bigskip

{\it Proof.}
Given $q_\tau^{k-1}\in \mathbb{Q}$ from the previous time step, suppose that $\{q_j\}\subset \mathbb{Q}$ is a minimizing sequence  for $q\mapsto\I(t_k,q)+\mathcal{D}(q_\tau^{k-1},q)$. 
The assumption (\ref{growth}) implies that $\{z_j\}$ is uniformly bounded in 
$W^{1,\beta}(\O;\R^{n\times n})\times W^{1,\omega}(\O;\R^m)$. 
Hence, as $\beta,\omega>1$ we can extract a weakly converging subsequence (not relabeled) $z_j\wto z$ in $W^{1,\beta}(\O;\R^{n\times n})\times W^{1,\omega}(\O;\R^m)$. The strong convergence of $z_j\to z:=(F_\p,p)$ in $L^\beta(\O;\R^{n\times n})\times L^\omega(\O;\R^m)$ ensures that $F_\p(x)\in {\rm SL}(n)$ almost everywhere. \jz Write $z_j=(F_\p^j,p^j)$, $q_j=(y^j,z_j)$. \EEE
Exploiting the submultiplicativity of the Euclidean norm, we estimate using Lemma~\ref{Yconseq}  
\begin{align}
\int_\O |\nabla y^j(x)(F_\p^j)^{-1}|^\alpha\,\md x&\ge\frac{ \|\nabla y^j\|_{L^d(\O;\R^{n\times n})}^\alpha}{\|F^j_\p\|_{L^\beta(\O;\R^{n\times n})}^\alpha}\nonumber\\
&\ge \jz\frac{\alpha}{d}\delta^{\alpha/(\alpha-d)}\|\nabla y^j\|_{L^d(\O;\R^{n\times n})}^d \EEE-\frac{\alpha-d}{d}\delta^{\alpha^2/(\alpha-d)^2}\|F^j_\p\|_{L^\beta(\O;\R^{n\times n})}^\beta.
\end{align}

The $L^d$-term on the right hand side is bounded due to (\ref{growth}) and the boundedness of $\{y^j\}$ in $W^{1,d}(\O;\R^n)$ follows by the Poincar\'{e} inequality if $\delta>0$ is taken small. Hence $y^j\wto y$ in $W^{1,d}(\O;\R^n)$ (up to a subsequence).  This then also implies that $\cof\nabla y^j\wto\cof\nabla y$ in $L^{d/(n-1)}(\O;\R^{n\times n})$; cf. \cite{ciarlet}.
Due to reflexivity of $W^{1,\alpha/(n-1)}(\O;\R^{n\times n})$ we get for a non-relabelled subsequence that 
$(\cof\nabla y^j )F^{j\top}_\p\wto \jz \,\Xi\EEE$ in  $W^{1,\alpha/(n-1)}(\O;\R^{n\times n})$ \jz for some $\Xi$ \EEE. However, the weak convergence of $\{\cof\nabla y^j\}_j$  and the strong convergence of 
$\{F_\p^j\}$ allow us to identify $\jz \Xi \EEE=(\cof\nabla y)F^\top_\p$. Moreover, the growth condition \eqref{growth-graddet1} implies that $\det\nabla y>0$ a.e. in $\Omega$ (see \cite{bbmkas}).
Cramer's rule together with $\det F_\p^j=1$ and $\det^{n-1}(\nabla y^j(F^j_\p)^{-1}) =\det(\cof\jz[\nabla y^j(F^j_\p)^{-1}]\EEE)= \det(\cof\nabla y^j)$ \jz give \EEE
$$
\frac{\cof[\nabla y^j (F_\p^j)^{-1}]}{\det\nabla y^j}=(\nabla y^j (F_\p^j)^{-1})^{-\top}$$
and the transpose of the left-hand side converges pointwise (for a subsequence again)  to $(\nabla y (F_\p)^{-1})^{-1}$ a.e. in $\O$. This, together with the fact that $\det\nabla y>0$ a.e.,  implies 
that $\nabla y^j (F^j_\p)^{-1}\to \nabla y (F_\p)^{-1}$ a.e. in $\O$. Then \jz by\cite[Corollary 7.9]{FL} we see \EEE that $\mathcal{I}$ is weakly lower semicontinuous.

The assumptions on $\mathcal{D}$ ensure that it is sequentially weakly lower semicontinuous on $\mathbb{Q}$, too. By the direct method of the calculus of variations, we conclude that a minimizer $q_\tau^k$ exists.
\hfill
$\Box$
\begin{remark}
Note that in fact, we proved that a minimizing sequence has a gpc-convergent subsequence.
\end{remark}

\subsection{Interpolation in time}

 We denote by $q_\tau$ a piecewise constant interpolation of $q_\tau^k=:(y_\tau^k,z^k_\tau)$, i.e. $q_\tau(t)=q_\tau^k$ if $t\in [k\tau,(k+1)\tau)$ and $k=0,\ldots, T/\tau-1$. Finally, $q_\tau(T):=q_\tau^N$. Analogously,
 $L_\tau(t,q_\tau(k\tau)):=L(k\tau,q_\tau(k\tau))$ is a piecewise constant interpolation of $L$ and $\I_\tau(t,q_\tau(k\tau)):=
 \I(k\tau,q_\tau(k\tau))$ is a piecewise constant  interpolation of $\I$.
 \begin{proposition}
  \label{prop:3}
  Under the assumptions of Proposition \ref{timestep-ex}, 
  problem~\eqref{timestep} has a solution $q_\tau(t)$ which is stable,
  i.e., for all $t \in [0,T]$ and for every $q \in \mathbb{Q}$,
  \begin{equation}
    \label{eq:stab}
    \G_\tau (t, q_\tau(t)) \le 
    \G_\tau(t, q) + \D \left(q_\tau(t), q \right).
  \end{equation}
  Moreover, for all $t_I \le t_{I\!I}$ from the set $\{k \tau\}_{k =
    0}^\NN$, the following discrete energy inequalities hold if one
  extends the definition of $q_\tau(t)$ by setting $q_\tau(t) \defeq
  q^0$ if $t < 0$:
  \begin{multline}
    \label{eq:energy}
    - \int_{t_I}^{t_{I\!I}} \dot L \left(t, q_\tau(t-\tau) \right) \d t \le
    \G \left(t_{I\!I}, q_\tau(t_{I\!I}) \right) 
    + {\rm Var}\left(\mathcal{D},q_\tau; \left[t_I, t_{I\!I} \right] \right) 
     - \G \left(t_I, q_\tau(t_I) \right) \\ 
     \le -\int_{t_I}^{t_{I\!I}} \dot L \left(t, q_\tau (t) \right) \d
    t. 
  \end{multline}
\end{proposition}

{\it Proof.} In  Proposition~\ref{timestep-ex} we proved the existence of a solution to~\eqref{timestep}.

To show the stability estimate~\eqref{eq:stab}, we use the minimizing
property of $q_\tau^k$ and a triangle inequality for $\D$. Indeed, since $q_\tau^k$ is a minimizer,
\begin{equation}
  \label{eq:17}
  \G \left(k \tau, q^k_\tau \right) + \D \left(q^{k-1}_\tau, q^k_\tau
  \right)  
  \le \G \left(k \tau, q \right) + \D \left(q^{k-1}_\tau,
  q \right),  
\end{equation}
from which we infer that
\begin{equation*}
   \G \left(k \tau, q^k_\tau \right)\le \G \left(k \tau, q \right) +
   \D \left(q^{k-1}_\tau, q \right)-\D \left(q^{k-1}_\tau, q^k_\tau
  \right).
\end{equation*}
However, the structure of the metric 
implies that 
\begin{equation*}
  \D \left(q^{k-1}_\tau, q \right) - \D \left(q^{k-1}_\tau,
  q^k_\tau \right) \le \D \left(q^{k}_\tau, q \right),
\end{equation*}
from which~\eqref{eq:stab} follows.

The next step is to verify energy inequality~\eqref{eq:energy}, following the thoughts of \cite{AMFTVL}.  Testing the stability of $q^{k-1}_\tau$ with $q \defeq q^{k}_\tau$, we get
\begin{align}\label{inequ}
  \G \left((k-1) \tau, q^{k-1}_\tau \right) & \le 
  \G \left((k-1) \tau, q^{k}_\tau \right) + \D \left( q^{k-1}_\tau, q^{k}_\tau
   \right)  \\
  &= \G \left(k \tau, q^{k}_\tau \right) + L \left(k \tau, q^{k}_\tau
  \right) - L \left((k-1) \tau, q^{k}_\tau \right)+ \D
  \left( q^{k-1}_\tau, q^{k}_\tau \right).\nonumber 
\end{align}
For nonnegative integers $k_1$, $k_2$ with $k_1 \le k_2 \le \NN$, let $t_I = k_1 \tau$ and $t_{I\!I} = k_2 \tau$. Summing ~\eqref{inequ} over $k = k_1 + 1,\dots, k_2$ gives
\begin{eqnarray}
  \label{eq:energy0}
  \sum_{k = k_1 + 1}^{k_2} 
  \left[ L \left((k-1) \tau, q^{k}_\tau \right) 
    - L \left(k \tau, q^{k}_\tau \right) \right]& \le&
  \G \left(k_2 \tau, q^{k_2}_\tau \right) 
  - \G \left(k_1 \tau, q^{k_1}_\tau \right)\\
& +&
  \sum_{k = k_1 + 1}^{k_2} \D \left(q^{k-1}_\tau, q^{k}_\tau \right).\nonumber 
\ee
Replacing $q_\tau^{k_1}$, $q_\tau^{k_2}$ with $q_\tau$ evaluated at suitable time points, we discover the first inequality in~\eqref{eq:energy},
\begin{align*}
  - \int_{t_I}^{t_{I\!I}} \dot L \left(t, q_\tau(t-\tau) \right) \d t
  & \le
  \G \left(k_2 \tau, q^{k_2}_\tau \right) 
  - \G \left(k_1 \tau, q^{k_1}_\tau \right)
  + \sum_{k = k_1 + 1}^{k_2} \D \left(q^{k-1}_\tau, q^{k}_\tau \right) \\
  & = \G \left(t_{I\!I}, q^{k_2}_\tau \right) 
      - \G \left(t_I, q^{k_1}_\tau \right)
      + {\rm Var} \left(\mathcal{D},q_\tau; \left[t_I, t_{I\!I} \right] \right) 
\end{align*}
(for a step function, calculating ${\rm Var} \left(\mathcal{D},q_\tau; \left[t_I, t_{I\!I} \right]\right)$ is easy). The proof of the second inequality in~\eqref{eq:energy} is similar. Starting from the minimality of $q^{k}_\tau$, when compared with $q^{k-1}_\tau$ in~\eqref{eq:17}, yields
\begin{multline*}
  \G \left(k \tau, q^{k}_\tau \right) 
  + \D \left(q^{k-1}_\tau, q^{k}_\tau \right) \le 
  \G \left(k \tau, q^{k-1}_\tau \right) \\ 
  = \G \left((k-1) \tau, q^{k-1}_\tau \right) 
  + L \left((k-1) \tau, q^{k-1}_\tau \right) 
  - L \left(k \tau, q^{k-1}_\tau \right).
\end{multline*}
We sum again over $k = k_1 + 1,\dots, k_2$ to obtain
\begin{eqnarray*}
  \G \left(k_2 \tau, q^{k_2}_\tau \right) 
 & -& \G \left(k_1 \tau, q^{k_1}_\tau \right) 
  + \sum_{k = k_1 + 1}^{k_2} \D \left(q^{k-1}_\tau, q^{k}_\tau \right)\\
 & \le& \sum_{k = k_1 + 1}^{k_2} \left[L \left((k-1) \tau, q^{k - 1}_\tau
  \right)
 - L \left(k\tau, q^{k - 1}_\tau \right) \right],
\end{eqnarray*}
so we find
\begin{equation*}
  \G \left(k_2 \tau, q^{k_2}_\tau \right) 
  - \G \left(k_1 \tau, q^{k_1}_\tau \right)
  +  {\rm Var} \left(\mathcal{D},q_\tau; \left[t_I, t_{I\!I} \right] \right) 
  \le - \int_{t_I}^{t_{I\!I}} \dot L \left(t, q_\tau(t) \right) \d t  
\end{equation*} 
and the second inequality in~\eqref{eq:energy} is shown.
\hfill
$\Box$



We would like to pass with the step size $\tau$ to zero and for this we need certain \emph{a priori} bounds.
\begin{proposition}
  \label{prop:4}
 Let (\ref{ass-f}) and (\ref{ass-g}) be satisfied.
Then there is $\kappa \in \R$ such that for any $\tau>0$:
\begin{equation}
  \label{first}
  \norm{y_\tau}_{L^\infty \left(0,T; W^{1,d}(\O;\R^n) 
  \right)} < \kappa ,
\end{equation}
\begin{equation}
  \label{second}
  {\rm Var}(\mathcal{D},q_\tau;[0,T])) < \kappa, 
\end{equation}  
\begin{equation}
\label{fourth}
\norm{z_\tau}_{L^\infty \left(0,T; W^{1,\alpha}(\O;\R^{n\times n})\times  W^{1,\beta}(\O;\R^{m}) 
  \right)} < \kappa.
\end{equation}
\end{proposition}
 
{\it Proof.} For $q=(y, F_\p,p)\in \mathbb{Q}$, set  
$$ V(q)= \int_\O \jz\mathcal{W}(\nabla y(x)F_\p^{-1},\nabla[(\cof\nabla y(x))F_\p^\top(x)],F_\p(x),\nabla F_\p,p(x),\nabla p(x))\,\md x\EEE. $$
Like in the proof of Proposition \ref{timestep-ex}, we can show by the growth conditions on $\mathcal{W}$ that for any $q$ with $V(q)<+\infty$ and some constants $C_0$, $C>0$,
\be\label{lowerbound} -C_0+C\left(\|y\|^d_{W^{1,d}(\O;\R^n)}+ \|F_\p\|^\beta_{W^{1,\beta}(\O;\R^{n\times n})}+\|p\|^\omega_{W^{1,\omega}(\O;\R^m)}\right)\le V(q).\ee 
Using a lower energy estimate from the proof of (\ref{eq:energy}) for $k_1=0$
we get 
\begin{equation*}
V(q^{k_2}_\tau)-L(k_2\tau,q^{k_2}_\tau)-V(q^0_\tau)+L(0,q^{0}_\tau)\le 
\sum_{k=1}^{k_2} \left[L \left((k-1) \tau, q^{k - 1}_\tau
  \right) -L \left(k\tau, q^{k - 1}_\tau \right) \right].
\end{equation*}
Hence for a constant $C$ which only depends on the data and the initial condition, 
\be\label{basic}
V(q^{k_2}_\tau)\le\sum_{k=1}^{k_2} \left[L \left((k-1) \tau, q^{k - 1}_\tau
  \right) -L \left(k\tau, q^{k - 1}_\tau \right) \right] +L(k_2\tau,q^{k_2}_\tau)+C.
\ee
So, denoting $Y_\tau^d:=\max_{1\le\ell\le T/\tau} \|y^{\ell}_\tau\|^d_{W^{1,d}(\O;\R^n)}$ and taking into account (\ref{lowerbound}) with $q:=q^{k_2}_\tau$ we have 
\be
Y_\tau^d\le C\sum_{k=1}^{k_2} \left[L \left((k-1) \tau, q^{k - 1}_\tau
  \right) -L \left(k\tau, q^{k - 1}_\tau \right) \right] +C(q^{k_2}_\tau).\ee
This gives the bound (\ref{first}), because $Y_\tau$ is raised to the power $d>1$ on the left hand side, whereas the right hand side is merely linear in $y_\tau^\ell$. We can proceed similarly for (\ref{fourth}). The upper bound on the total dissipation can be derived by algebraic manipulations as in \cite{fm}.
\hfill
$\Box$

\bigskip

\subsection{Limit passage}

The following lemma is proved in \cite{mami1}.
  
\begin{lemma}\label{helly} Let $\mathcal{D}\colon\mathbb{Z}\times\mathbb{Z}\to[0,+\infty]$ \jz satisfy (\ref{assonD1}) and (\ref{assonD2})\EEE. Let $\mathcal{K}$ be a weakly sequentially compact subset of $\mathbb{Z}$. Then for every sequence $\{z_k\}_{k\in\N}$, $z_k\colon[0,T]\to \mathcal{K}$ for which
$\sup_{k\in\N}\,{\rm Var}(\mathcal{D},z_k;[0,T])<C$, $C>0$, there exists a subsequence (not relabelled), a function $z\colon[0,T]\to\mathcal{K}$, and a function $\Delta\colon[0,T]\to [0,C]$ such that:\\
\noindent
(i) ${\rm Var}(\mathcal{D},z_k;[0,t])\to \Delta(t)$ for all $t\in[0,T]$,\\
\noindent
(ii) $z_k\wto z$ in $\mathbb{Z}$ for all $t\in[0,T]$, and\\
\noindent
(iii) ${\rm Var}(\mathcal{D},z;[t_0,t_1])\le\Delta(t_1)-\Delta(t_0)$ for all $0\le t_0<t_1\le T$.
\end{lemma}

\bigskip
Let us denote  $\mathbb{X}:= L^{\beta}(\O;\R^{n\times n})\times L^\omega(\O;\R^m)$.
Finally, we proved the existence of an energetic solution.

\bigskip

\begin{theorem}\label{existence}
Let $\alpha^{-1}+\beta^{-1}\le d^{-1}<\jz (n-1)^{-1}\EEE$, $\jz d>\frac{\beta(n-1)}{\beta-1}\EEE$. Let $q^0\in\mathbb{Q}$ be a stable initial condition. Let the assumptions on $\mathcal{W}$, $\mathcal{D}$, $f$ and $g$  from Section~\ref{assump} hold. Let further (\ref{nodi1}) or  (3.A), (3.B), (3.C) hold.
Then there is  a process $q\colon[0,T]\to \mathbb{Q}$ 
with $q(t)=(y(t),z(t))$  such that $q$ is an energetic solution according to Definition~\ref{en-so}.
The following convergence statements also hold:\\
 \noindent
 (i)
for a $t$-dependent (not relabelled) subsequence  w-$\lim_{\tau\to 0}y_{\tau}(t)=y(t)$ in $W^{1,d}(\O;\R^n)$ for all $t\in[0,T]$, \\   
 \noindent
(ii) for a (not relabelled) subsequence $\lim_{\tau\to 0} z_{\tau}(t)=z(t)$ in $\mathbb{X}$ for all $t\in[0,T]$, \\
\noindent
 (iii) for a (not relabelled) subsequence
 $\lim_{\tau\to 0}\I_{\tau}(t,q_{\tau}(t))=\I(t,q(t))$ for all $t\in[0,T]$, and\\
\noindent
(iv) for a (not relabelled) subsequence $\lim_{\tau\to 0} {\rm Var}(\mathcal{D},q_\tau;[0,t])= {\rm Var}(\mathcal{D},q;[0,t])$ for all $t\in[0,T]$.

\end{theorem}

\proof
We have adapted the proof \jz from \cite{fm, mielke2} \EEE and divided it into \jz three \EEE steps.

\noindent
{\it Step 1}:
\jz Assertion (i) follows from the \textit{a priori} estimate in Proposition~\ref{prop:4} and (ii) results from Lemma~\ref{helly}. 
Hence \EEE we have a limit $q(t)=(y(t),z(t))$. We easily get that $q(t)\in\mathbb{Q}$ for all $t\in[0,T]$.  

Put $S(t,\tau):=\max_{k\in\N\cup\{0\}}\{k\tau;\ k\tau\le  t\}$. Then $\lim_{\tau\to 0}S(t,\tau)=t$  and by definition $q_\tau(t):=q_\tau(S(t,\tau))\in\mathcal{S}(S(t,\tau))$. As we have seen, $\mathcal{D}$ can be assumed such that (\ref{limsup}) holds. Hence by Proposition~\ref{limstab}, the limit is stable, i.e. $q(t)\in \mathcal{S}(t)$ (thanks to our \textit{a priori} estimates, we can pass to a gpc-convergent subsequence to get lower semicontinuity).
\\

\noindent
{\it Step 2}:
\jz Notice that $\theta_{\tau}(t):= \frac{\partial L}{\partial t}(t,q_{\tau}(t))$ is bounded in $L^\infty(0,T)$ by (\ref{ass-f}), (\ref{ass-g}) and (\ref{first}), so we deduce a weak*-convergent subsequence, which we do not relabel, with the limit $\theta$.  By Fatou's lemma $ \theta_{\rm i}\le\theta$. \EEE

Our interpolation satisfies $q_\tau(t)=q_\tau(k\tau)$ for $0\le t-k\tau< \tau$. Successively using (\ref{eq:energy}), (\ref{ass-f}), (\ref{ass-g}) and (\ref{first}), we get for some $C$, $C_1>0$ 
\begin{eqnarray*}
\I(t,q_\tau(t))+{\rm Var}(\mathcal{D},q_\tau;[0,t])&\le& \I(k\tau,q_\tau(k\tau))+{\rm Var}(\mathcal{D}, q_\tau;[0,k\tau]) +C\tau\\
&\le& \I(0,q_\tau(0))-\int_0^{k\tau}\dot L(s,q_\tau(s))\,\md s +C\tau \\
&\le& \I(0,q_\tau(0))-\int_0^t \dot L(s,q_\tau(s))\,\md s+C_1\tau.
\end{eqnarray*} 

  On account of Lemma~\ref{helly} (i) and Proposition \ref{limstab} we find that for $\tau\to 0$
\begin{eqnarray*}
\I(t,q(t))+ \Delta(t) &\le& \I(0,q(0))-\int_0^t \theta(s)\,\md s.\end{eqnarray*}

\jz Following \cite{mielke2} we set $\theta_{\rm i}(t):=\liminf_{\tau\to 0}\theta_{\tau}(t)$. \EEE Since $\Delta(t)\ge{\rm Var}(\mathcal{D},q;[0,t])$ and by Fatou's lemma  $\int_0^t\theta(s)\,\md s\ge\int_0^t\theta_{\rm i}(s)\,\md s$, for a.a. $t\in[0,T]$ we get 
\begin{eqnarray*}\I(t,q(t))  +{\rm Var}(\mathcal{D},q;[0,t]) &\le& \I(0,q(0))-\int_0^t \theta_{\rm i}(s)\,\md s.\end{eqnarray*}

The convergence $\theta_\tau(s)=\dot L(s,q_\tau(s))\to\dot L(s,q(s))$ and extracting an $s$-dependent subsequence of $\{\theta_\tau(s)\}$ converging to $\theta_{\rm i}(s)$ show that $\theta_{\rm i}(s)=\dot L(s,q(s))$.
We finally get the upper energy estimate 
\be
\label{upperest}
\I(t,q(t)) +{\rm Var}(\mathcal{D},q;[0,t]) &\le& \I(0,q(0))-\int_0^t \dot L(s,q(s)\,\md s.
\ee

\jz\textit{Step 3:} \EEE The fact that $q(t)$ is stable for all $t\in[0,T]$ will be useful in proving the lower estimate. Consider a partition of a time interval $[t_I,t_{I\!I}]\subset [0,T]$ such that $t_I=\vartheta_0^M<\vartheta_1^M<\vartheta_2^M<\cdots<\vartheta_{K_M}^M=t_{I\!I}$ and $\max_{i} (\vartheta_i^M-\vartheta_{i-1}^M)=:\vartheta^M\to 0$ as $M\to\infty$. Let us test the stability of $q(\vartheta_{k-1}^M)$ with $q(\vartheta_k^M)$, $k=1,\ldots, K_M$.
Similarly to (\ref{eq:energy0}) we get  
\be
  \label{eq:energy00}
  \sum_{k = 1}^{K_M} 
  \left[L\((\vartheta_{k-1}^M, q(\vartheta_k^M)\)-L\(\vartheta_k^M,q(\vartheta_k^M)\)\right]&\le& 
  \G \(t_{I\!I}, q(t_{I\!I})\) - \G \(t_I, q(t_I)\)\\
& +& \sum_{k = 1}^{K_M} \D \(q(\vartheta_{k-1}^M), q(\vartheta_k^M)\).\nonumber 
\ee
  
Thus, 
\be
  \label{eq:energy000}
  \sum_{k =  1}^{K_M} 
  -\int_{\vartheta_{k-1}^M}^{\vartheta_k^M}\dot L(s, q(\vartheta_k^M))\,\md s&\le& 
  \G \(t_{I\!I}, q(t_{I\!I}\) - \G \(t_I, q(t_I)\)\nonumber\\
& +& {\rm Var} (\mathcal{D},q;[t_I,t_{I\!I}]).
\ee

We finally \jz rearrange the terms as \EEE 
\be\label{eq:forces}\sum_{k =1}^{K_M} \int_{\vartheta_{k-1}^M}^{\vartheta_k^M}\dot L(s, q(\vartheta_k^M))\,\md s&=& \sum_{k = 1}^{K_M} \dot L(\vartheta_k^M,q(\vartheta_k^M))(\vartheta_k^M-\vartheta_{k-1}^M)\nonumber\\
& +& \sum_{k = 1}^{K_M}\int_{\vartheta_{k-1}^M}^{\vartheta_k^M}(\dot L(s, q(\vartheta_k^M))-\dot L(\vartheta_k^M,q(\vartheta_k^M)))\,\md s. \ee

The last sum on the right-hand side of (\ref{eq:forces}) goes to zero with $\vartheta^M\to 0$ because the time derivative of the external loading is uniformly continuous in time by (\ref{ass-f}) and (\ref{ass-g}).  
The first sum on the right-hand side converges to $\int_{t_I}^{t_{I\!I}}\dot L(s,q(s))\,\md s$ by \cite[Lemma~4.12]{dmgfrt}. 

Altogether, the upper and lower estimates give us the energy balance 

\be
\label{e-balance}
\I(t,q(t)) +{\rm Var}(\mathcal{D},q;[0,t]) &=& \I(0,q(0))-\int_0^t \dot L(s,q(s)\,\md s.
\ee

Moreover, 
\be\label{improvedconv}
& &\I(0,q(0)) -\int_0^t\theta_{\rm i}(s)\,\md s\le \I(t,q(t)) +{\rm Var}(\mathcal{D},q;[0,t]))\le \I(t,q(t)) +\Delta(t)\nonumber\\
&\le& \I(0,q(0)) -\int_0^t\theta(s)\,\md s\le \I(0,q(0)) -\int_0^t\theta_{\rm i}(s)\,\md s.
\ee
Thus in fact, all inequalities in (\ref{improvedconv}) are equalities and we get (iv). \jz Proposition~\ref{limstab} also implies (iii). \EEE
\hfill\qed

\bigskip

Finally, we include a simple example from single-crystal plasticity, covered by our approach, in the spirit of \cite{gurtin} and \cite{amLieGroups}.
 
\begin{example}[single slip without hardening]
Consider two orthonormal vectors $a$, $b\in\R^3$, describing the motion of edge dislocations in such a way that $a$ is the glide direction and $b$ is the slip-plane normal.
Further we focus on a particular case of the so-called {\it separable} material 
where 
$$\mathcal{W}(F_\e, H, z)=\mathcal{W}_1(F_\e)+c|H|^2+c\det(F_\e)^{-s}+\varepsilon|F_\p|^6+\varepsilon|\nabla F_\p|^6,\;c,s,\varepsilon>0,$$
with $F_\p(x,t)=\mathbb{I}+\gamma(x,t)a\otimes b$ where $\gamma$ is the plastic slip and we can choose $\mathcal{W}_1$ to be e.g. the stored energy density of the Saint Venant–Kirchhoff material, i.e. $\mathcal{W}_1(F_\e)=\frac{1}{8}(\T{F_\e}F_\e-\mathbb{I}):\mathbb{C}:(\T{F_\e}F_\e-\mathbb{I})$ so that $\mathcal{W}_1(F_\e)\ge c(|F_e|^4-1)$.
The vectors $a$, $b$ are not fixed in the reference configuration in general (more precisely, they lie in an intermediate lattice space, \jz see \EEE \cite{gurtinBook}). The slip-plane normal $\tilde b$ in the reference configuration is given by $\tilde b= (F_\p)^\top b$. However, in the special case of a single slip we obtain $\tilde b=b$ and the slip-plane normal remains unchanged during plastic transformations.

As $F_\p$ is completely described by $\gamma$, we identify $z:=\gamma$ and use the dissipation potential
$$
\delta(\dot\gamma)=\kappa|\dot\gamma|,
$$
where $\kappa>0$ represents the resistance to the slip.

This corresponds to the dissipation distance 
$$
\mathcal{D}(\gamma_1,\gamma_2)=\int_\Omega \kappa|\gamma_1(x)-\gamma_2(x)|\md x.
$$
An extension to multi-slip, described by several glide directions $\{a_i\}$ and slip-plane normals $\{b_i\}$, $1\le i\le N$, would be feasible, cf. \cite{amLieGroups}.
\end{example}
As a last remark, note that our stored energy $\mathcal{W}$ could also depend on \mk  $F_{\rm e}^{-1}$, \EEE cf. \cite{bbmkas}.

\bigskip


\bigskip

{\bf Acknowledgements.} 
\mk This work was  supported by the GA\v{C}R project 18-03834S (MK \& JZ)  and by    the DFG Priority Programme (SPP) 2256 (JZ).

\end{document}